\documentclass[11pt]{amsart}
\title{The Projective Envelope of a Cuspidal Representation of a Finite Linear Group}
\author{David Paige}

\usepackage{epsfig}
\usepackage{graphicx}
\usepackage{amsfonts}
\usepackage{amssymb}
\usepackage{eucal}
\usepackage{mathrsfs}
\usepackage{amsthm}
\usepackage{stmaryrd}
\usepackage{amsmath}
\usepackage{latexsym}
\usepackage{amscd}
\usepackage{lscape}
\usepackage{amsrefs}
\usepackage[all]{xy}
\usepackage{enumitem}
\usepackage{bbm}
\usepackage{multicol}
\usepackage{tikz}
\usepackage{textcomp}
\usepackage{fix-cm}

\usetikzlibrary{matrix}

\input amssym.def

\newcommand{\N}{\mathbb{N}}

\newcommand{\Q}{\mathbb{Q}}
\newcommand{\Z}{\mathbb{Z}}
\newcommand{\C}{\mathbb{C}}

\newcommand{\F}{{\mathbb F}}

\renewcommand{\O}{\mathcal{O}}

\newcommand{\m}{\mathfrak{m}}

\newcommand{\W}{\mathrm{W}}

\newcommand{\Fr}{\mathrm{Fr}}
\newcommand{\id}{\mathrm{id}}

\newcommand{\Hom}{\mathrm{Hom}}

\newcommand{\Gal}{\mathrm{Gal}}

\newcommand{\End}{\mathrm{End}}
\newcommand{\GL}{\mathrm{GL}}
\newcommand{\Gl}{\mathrm{GL}}

\newcommand{\Tr}{\mathrm{Tr}}

\newcommand{\Ind}{\mathrm{Ind}}
\newcommand{\Irr}{\mathrm{Irr}}

\newcommand{\Rep}{\mathrm{Rep}}

\newcommand{\ord}{\mathrm{ord}}

\newcommand{\diag}{\mathrm{diag}}

\newcommand{\St}{\mathrm{St}}

\newcommand{\mc}[1]{\mathcal{#1}}
\newcommand{\mf}[1]{\mathfrak{#1}}
\newcommand{\ms}[1]{\mathscr{#1}}

\newtheorem{thm}{Theorem}[section]
\newtheorem*{thm*}{Theorem}
\newtheorem{cor}[thm]{Corollary}
\newtheorem{lem}[thm]{Lemma}
\newtheorem{pro}[thm]{Proposition}
\theoremstyle{definition}
\newtheorem{defn}[thm]{Defintion}
\theoremstyle{remark}
\newtheorem{rem}{Remark}

\numberwithin{equation}{section}

\begin{document}

\bibliographystyle{plain}

\maketitle

\begin{abstract} Let $\ell$ be a prime and let $q$ be a prime power not divisible by $\ell$.  Put $G=\GL_n(\F_q)$ and fix an irreducible cuspidal representation, $\bar{\pi}$, of $G$ over a sufficiently large finite field, $k$, of characteristic $\ell$ such that $\bar{\pi}$ is not supercuspidal.  We compute the $\W(k)[G]$-endomorphism ring of the projective envelope of $\bar{\pi}$ under the assumption that $\ell>n$.  Our computations provide evidence for a conjecture of Helm relating the Bernstein center to the deformation theory of Galois representations (see \cite{HelmNew}).\end{abstract}

\tableofcontents

\section{Introduction}

\label{S:Introduction}

Let $p$ and $\ell$ be distinct primes and suppose that $F$ is a $p$-adic field.  One of the most celebrated results of modern number theory is the ($\ell$-adic) local Langland correspondence on $F$. The local Langlands correspondence is a canonical well-behaved bijection $$\Irr_\ell(\GL_n(F))\to\Rep^n_\ell(\W_F),$$ where $\Rep^n_\ell(\W_F)$ is the set of continuous $n$-dimensional Frobenius semi-simple representations of the Weil group, $\W_F$, of $F$ over $\bar{\Q}_{\ell}$ and $\Irr_\ell(\GL_n(F))$ is the set of irreducible admissible representations over $\bar{\Q}_\ell$ of the general linear group, $\GL_n(F)$ (see \cite{LLC}, \cite{LLCFrench} and Section 32 of Chapter 7 of \cite{LLC2D}).

Certainly, part of the importance of the local Langlands correspondence is that notions on one side often correspond nicely with notions on the other.  In particular, the notion of semi-simplification on the Weil side corresponds with the notion of supercuspidal support on the general linear side, a bijection known as the ($\ell$-adic) semi-simple Langlands correspondence.

The semi-simple local Langlands correspondence is important partially because, unlike the local Langlands correspondence itself, it translates well to the case of modular representations.  Indeed, if $k$ is the residue field of $\bar{\Q}_\ell$, Vign\'{e}ras has shown the following:

\begin{thm*} There is a unique bijection between supercuspidal supports of $\GL_n(F)$-representations over $k$ and $n$-dimensional semi-simple Weil representations over $k$ that is compatible with the semi-simple local Langlands correspondence and reduction modulo $\ell$.\end{thm*}

\begin{proof} See \cite{ModLLC}.\end{proof}

For obvious reasons, Vign\'{e}ras's correspondence is often called the $\ell$-modular semi-simple local Langlands correspondence.

One is led to consider whether these correspondences can be understood geometrically.  To fix ideas, let $\bar{\pi}$ be an irreducible representation of $\GL_n(F)$ over $k$ such that $\bar{\pi}$ is cuspidal but not supercuspidal.  Denote by $\bar{\rho}$ the semi-simple Weil representation over $k$ corresponding to $\bar{\pi}$ via the $\ell$-modular semi-simple local Langlands correspondence.  Attached to $\bar{\rho}$ is the framed universal deformation ring, $R_{\bar{\rho}}^\boxempty$, which parameterizes lifts of $\bar{\rho}$ together with a choice of basis.  One is led to consider whether a corresponding algebraic object can be found on the general linear side of local Langlands.

An important objection to this end is the \emph{Bernstein center}, which for any category $\mc{A}$, is the endomorphism ring of the identity functor on $\mc{A}$.  Classically, the Bernstein center has been important in the study of representations of $\GL_n(F)$.  In particular, Bernstein and Deligne were able calculate the center of the category, $\Rep_\C(\GL_n(F))$, of smooth $\C$-representations of $\GL_n(F)$ (see \cite{BD}).  Moreover, they give a decomposition of the category into a product of blocks (called Bernstein components) and a description of the center of each block, which they show to be a finitely generated $\C$-algebra.  These results can also be translated to the field $\bar{\Q}_\ell$.

In consideration of a geometric interpretation of the local Langlands correspondence, Helm has considered the situation for representations over the Witt vectors, $\W(k)$.  In particular, he obtains a block decomposition of the category $\Rep_{\W(k)}(\GL_n(F))$ which is analogous to, but coarser than, that obtained by Bernstein and Deligne (Theorem 10.8 of \cite{Helm}).  Denote by $\mc{C}_{\bar{\pi}}$ the block corresponding to $\bar{\pi}$.  Then if $A_{\bar{\pi}}$ is the Bernstein center of $\mc{C}_{\bar{\pi}}$, Helm shows that $A_{\bar{\pi}}$ is finitely generated as an $\W(k)$-algebra (Theorem 12.1 of \cite{Helm}).

Given any irreducible representation in $\mc{C}_{\bar{\pi}}$, Schur's lemma gives a map $A_{\bar{\pi}}\to k$ and Helm shows (Theorem 12.2 of \cite{Helm}) that this relationship induces a bijection between maps $A_{\bar{\pi}}\to k$ and possible supercuspidal supports of representations in $\mc{C}_{\bar{\pi}}$.  In particular, the supercuspidal support of $\bar{\pi}$ gives a maximal ideal, $\mf{m}_{\bar{\pi}}$, of $A_{\bar{\pi}}$.  Helm has conjectured that there is an isomorphism $$(A_{\bar{\pi}})_{\mf{m}_{\bar{\pi}}}\to (R_{\bar{\rho}}^\boxempty)^{\GL_n}$$ (where the $\GL_n$ superscript denotes invariants under the change of basis action of $\GL_n$ and the $\mf{m}_{\bar{\pi}}$ denotes completion at $\mf{m}_{\bar{\pi}}$), which is compatible with the local Langlands correspondence (see Conjecture 7.5 of \cite{HelmNew}).  It is therefore advantageous to study the structure of the ring $A_{\bar{\pi}}$.

To this end, attached to $\bar{\pi}$ are three integers $d$, $e$, and $f$ and a particular representation, $\bar{\sigma}$, of $\GL_{\frac{n}{ef}}(\F_{q^f})$ (by the theory of types: see IV.1.1-IV.1.3 and IV.3.1B of \cite{MFV2}).  We denote by $P_{\bar{\sigma}}\to\bar{\sigma}$ the projective envelope of $\bar{\sigma}$ in the category of representations over $\W(k)$.  There is also a unique characteristic zero representation, $\St_{\bar{\sigma}}$, such that $\St_{\bar{\sigma}}$ is not supercuspidal but an $\ell$-modular reduction of $\St_{\bar{\sigma}}$ contains a copy of $\bar{\sigma}$. $\St_{\bar{\sigma}}$ is a so-called generalized Steinberg representation.  Moreover, if $L$ is a sufficiently large finite extension of the field of fractions of $\W(k)$, $P_{\bar{\sigma}}\otimes L$ contains a unique summand isomorphic to $\St_{\bar{\sigma}}$ and we denote by $I_0$ the set of endomorphisms of $\End(P_{\bar{\sigma}})$ that annihilate this summand. Helm has then shown that $$A_{\bar{\pi}}=\End(P_{\bar{\sigma}})\left[T_1,\ldots,T_{\frac{n}{def}}^{\pm1}\right]/\Big<T_1,\ldots,T_{\frac{n}{def}-1}\Big>\cdot I_0$$ (see Lemmas 4.7 and 4.8, Proposition 7.21, and Corollaries 10.19 and 11.11 of \cite{Helm}).

In this paper, under the assumption that $\ell>n$, we calculate the ring $\End(P_{\bar{\sigma}})$, motivated by its role in determining the structure of $A_{\bar{\pi}}$ and its relation to $R_\rho^\boxempty$.  We note that the cuspidal representation $\bar{\sigma}$ comes attached with a notion of degree (see Definition \ref{CuspDef}). Explicitly, our main result is the following (see Theorem \ref{Main}).

\begin{thm*} Suppose that $k$ is a finite field of characteristic $\ell$. Let $q$ be a power of a prime distinct from $\ell$. Denote the order of $q$ modulo $\ell$ by $w$ and put $r=\ord_\ell(q^w-1)$. Suppose that $\bar{\sigma}$ is an irreducible cuspidal representation of $G=\GL_n(\F_q)$ over $k$ of degree $d<n$ and let $P_{\bar{\sigma}}$ denote the projective envelope of $\bar{\sigma}$ as a $\W(k)[G]$-module.  Then, under the assumptions that $2\leq n<\ell$ and that $k$ is large enough to contain the $\ell$-regular $|G|$th roots of unity,  $\End(P_{\bar{\sigma}})$ is isomorphic to the subring of $$\W(k)[X]/(X^{\ell^r}-1)$$ of invariants under the map $X\mapsto X^{q^d}$.\end{thm*}

Moreover, we will give our isomorphism explicitly by finding a generator of the ring of invariants of $$\W(k)[X]/(X^{\ell^r}-1)$$ and giving its action on the direct summands of $P_{\bar{\sigma}}\otimes_{\W(k)}L$ (see the remark following Theorem \ref{Main}).

Finally, in Section \ref{S:Helm}, we will show that the results of our calculations are compatible with the conjecture of Helm discusses.  More explicitly, we will find, using our calculations, a copy of $A_{\bar{\pi}}$ in $(R_{\bar{\rho}}^\boxempty)^{\GL_n}$ (since $(R_{\bar{\rho}}^\boxempty)^{\GL_n}$ is complete, this will imply that there is also a copy of $(A_{\bar{\pi}})_{\m_{\bar{\pi}}}$).

\section{A Characterization of the Projective Envelope}

\label{S:ProjEnv} For the rest of the paper, we fix a prime $\ell$ and a finite field, $k$, of characteristic $\ell$.  We denote by $K$ the field of fractions of the Witt Vectors, $\W(k)$, of $k$.  We also fix a prime power $q$ that is not divisible by $\ell$ and an integer $n\geq 2$.  Put $G=\GL_n(\F_q)$ and choose an algebraic closure of $\F_q$ so that we may speak of $\F_{q^d}$ for any $d$.  We denote by $L$ a finite extension of $K$ which is large enough to admit all of the characteristic zero representations of $G$ and chose an identification between $\F_{q^{n!}}$ and its character group over $L$.  We assume that $k$ is large enough so that $L$ may be chosen to have residue field $k$.  We also assume that $\ell>n$.

In this section, we will first reduce our work to the consideration of a certain special case (namely that of a cuspidal representation associated to $1\in\F_{q^n}^\times$ in a sense given below).  We will then, following \cite{Helm}, give a nice characterization for the projective envelope in this case (though the analogous characterization works in general).  Finally,  we will outline our strategy for computing the corresponding ring of endomorphisms.  We remark that our strategy can be employed in this general case, but passing to the simple case greatly simplifies the calculations (for the details of the general calculation, see \cite{thesis}).

Since $G$ is a reductive group, Deligne-Lusztig Theory, together with our identification between $\F_{q^n}^\times$ and its character group, gives a canonical bijection between the irreducible representations of $G$ over $L$ (respectively $k$) and the conjugacy classes of $G$ (respectively the $\ell$-regular conjugacy classes) For a summary of this construction, see Section 5 of \cite{Helm} or Section III.2 of \cite{MFV}.

In particular, if $s\in\F_{q^n}$ is the $\ell$-regular part of an element of $\F_{q^n}$ of degree $n$ over $\F_q$, we have a corresponding cuspidal representation $\bar{\pi}_s$ over $k$.  Here we are of course representing a semi-simple conjugacy class in $G$ with a field element by choosing one of its eigenvalues (such a choice being valid up to $\F_q$-Galois conjugacy).  All the cuspidal representations of $G$ over $k$ have the form $\bar{\pi}_s$ for some such $s$.

\begin{defn}\label{CuspDef} Let $\bar{\pi}_s$ be a cuspidal representation of $G$ over $k$.   The \emph{degree} of $\bar{\pi}_s$ is the degree of $s$ over  $\F_q$.\end{defn}

Given such an $s$, we may consider the collection of irreducible representations of $G$ over $k$ with the same supercuspidal support as $\bar{\pi}_s$.  This collection is a union of blocks (see Theorem 5.6 of \cite{Helm}) and we will denote by $\mc{C}_s$ the corresponding full subcategory of the category of $\W(k)[G]$-modules.

\begin{thm}\label{Morita}  Suppose that $\bar{\pi}_s$ is a cuspidal representation of $G$ of degree $d<n$.  Put $G'=\GL_{n/d}(\F_{q^d})$.  Denote by $\bar{\pi}'_1$ the cuspidal representation of $G'$ coming from $1\in\F_{q^d}^\times$ and let $\mc{C}'$ be the corresponding full subcategory of the category of $\W(k)[G']$-modules.  Then we have an equivalence of categories $\mc{C}_s\to\mc{C}'$ with $\bar{\pi}_s\mapsto\bar{\pi}'_1$.  \end{thm}

\begin{proof}   By assumption, $s$ is the $\ell$-regular part of a generator of $\F_{q^n}/\F_{q^d}$.  We may thus consider the cuspidal representation $\bar{\pi}'_s$ of $G'$ and the associated subcategory $\mc{C}'$.  Theorem B' of Section 11 of \cite{Bonnafe} gives an equivalence of categories $\mc{C}_s\to\mc{C}_s'$ that takes $\bar{\pi}_s$ to $\bar{\pi}_s'$.  The equivalence named in the theorem comes after twisting by $\chi_s$, where $\chi_s$ is the character of $\F_d^\times$ associated to $s$.\end{proof}

Theorem \ref{Morita} reduces our work to the case that our cuspidal representation is $\bar{\pi}=\bar{\pi}_1$.  We remark that $n$ must be the order of $q$ modulo $\ell$ and that the supercuspidal support of $\bar{\pi}_1$ is $n\bar{\mathbbm{1}}$, where $\bar{\mathbbm{1}}$ is the trivial representation of $\F_q^\times$ (see Part (b) of the corollary of Section III.2.5 of \cite{MFV}).

To give our characterization of the projective envelope of $\bar{\pi}$, we let $U$ be the subgroup of $G$ consisting of strictly upper-triangular matrices and let $\psi:U\to\W(k)$ be a fixed generic character (choose a nontrivial homomorphism $\psi':\F_q\to\W(k)^\times$ and define the image under $\psi$ of an element in $U$ to be the value under $\psi'$ of the sum of the values of the off-diagonal entries).  Recall that an irreducible representation is generic if and only if it contains a vector on which $U$ acts by $\psi$.

We also denote by $e_{\bar{\pi}}$ the idempotent in $L[G]$ which is the sum of all the irreducible representations of $G$ over $L$ with supercuspidal support $n\mathbbm{1}$.

\begin{pro} The idempotent $e_{\bar{\pi}}$ lies in $\W(k)[G]$ and the representation $P_{\bar{\pi}}=e_{\bar{\pi}}\Ind_U^G\psi$ is a projective envelope for $\bar{\pi}$ in the category of $\W(k)[G]$-modules.\end{pro}

\begin{proof} See Theorem 5.6 and Proposition 5.6 of \cite{Helm}.\end{proof}

The idempotent $e_{\bar{\pi}}$ thus corresponds to a direct summand of the category of $\W(k)[G]$-modules.  We denote this full subcategory by $\mc{C}$. We also note that there is a unique characteristic zero generic representation, $\St_L$, which supercuspidal support $n\mathbbm{1}$.  In fact, $\St_L$ is the so-called \emph{Steinberg representation} of $G$ which is defined for any reductive algebra group over a finite or local field (see \cite{SteinHis}).

\begin{pro}\label{CharZeroProjEnv} The representation $P_{\bar{\pi}}\otimes_{\W(k)}L$ is the direct sum of a single copy of the Steinberg representation, $\St_L$, and a single copy of each of the supercuspidal representations that lift $\bar{\pi}$.\end{pro}

\begin{proof} See Corollary 5.14 of \cite{Helm}.\end{proof}

Since $P_{\bar{\pi}}$ is projective, we have an injection $$\End_{\W(k)[G]}(P_{\bar{\pi}})\hookrightarrow\End_{L[G]}(P_{\bar{\pi}}\otimes_{\W(k)}L).$$ In particular, $\End_{\W(k)[G]}(P_{\bar{\pi}})$ is commutative.

To complete our strategy for computing $\End_{\W(k)[G]}(P_{\bar{\pi}})$, we consider the Bernstein center of the category $\mc{C}$.  Recall from Section \ref{S:Introduction} that the \emph{Bernstein center} of an abelian category $\mc{A}$ is the endomorphism ring of the identity functor $\id:\mc{A}\to\mc{A}$.  In other words, an element of the Bernstein center is a choice of endomorphism on each object of $\mc{A}$ which commutes with any morphism of $\mc{A}$ in the obvious sense.  In particular, the Bernstein center of $\mc{C}$ is $e_{\bar{\pi}}Z(\W(k)[G])$.

\begin{pro} Suppose that $Q$ is a faithfully projective object in the abelian category $\mc{A}$.  Then $M\mapsto\Hom(P,M)$ is an equivalence of categories from $\mc{A}$ to the category of right $\End_{\mc{A}}(Q)$-modules.\end{pro}

\begin{proof} See \cite{Roche}, Theorem 1.1.\end{proof}

In particular, we have an isomorphism from the Bernstein center of $\mc{A}$ to $Z(\End_{\mc{A}}(Q))$.  The map takes an element of the Bernstein center to its action on $Q$.  In order to use this last result to our advantage, we construct a faithfully projective object in the category $\mc{C}$.

\begin{lem}\label{CuspSupp} Every simple object in $\mc{C}$, aside from $\bar{\pi}$, has cuspidal support $n\bar{\mathbbm{1}}$.\end{lem}

\begin{proof} Let $\bar{\pi}'$ be a simple object in $\mc{C}$.  A representation, $\bar{\pi}_0$, appearing in the cuspidal support for $\bar{\pi}'$ must itself have supercuspidal support equal to $m\mathbbm{1}$, for some $m<n$.  Hence $\bar{\pi}_0$ is the representation of $\GL_m(\F_q)$ coming from 1.  We conclude that an $\ell$-power torsion element generates $\F_{q^m}/\F_q$.  Thus we have $m=n$ or $m=1$.  If any representation in the cuspidal support gives $m=n$, we must have $\bar{\pi}'=\bar{\pi}$.  Otherwise, the cuspidal support of $\bar{\pi}'$ is as claimed.\end{proof}

If $\bar{\pi}_0$ is the parabolic induction of $n\bar{\mathbbm{1}}$, Lemma \ref{CuspSupp} implies that $\bar{\pi}_0$ surjects onto any simple object in $\mc{C}$ other than $\bar{\pi}$.  Hence, if we let $\pi_0$ be the parabolic induction of $n$ copies of the trivial representation over $\W(k)$, we see that $\pi_0$ is projective and surjects onto all the simple objects in $\mc{C}$ other than $\bar{\pi}$.

Putting $P_0=e_{\bar{\pi}}\pi_0$, we conclude that $Q_{\bar{\pi}}=P_0\oplus P_{\bar{\pi}}$ is a faithfully projective object in $\mc{C}$.  Since $Q_{\bar{\pi}}$ is constructed as a direct product, its  $\W(k)[G]$-endomorphism ring can be represented as the collection of block matrices $$\left(\begin{matrix}\End(P_0) & \Hom (P_{\bar{\pi}},P_0)\\ \Hom(P_0,P_{\bar{\pi}}) & \End(P_{\bar{\pi}})\end{matrix}\right).$$ In particular we get an embedding $$Z(\End(Q_{\bar{\pi}}))\hookrightarrow Z(\End(P_0))\times Z(\End(P_{\bar{\pi}})).$$

The relevant point to all of this is the following proposition:

\begin{pro} The map $Z(\End_{\W(k)[G]}(Q_{\bar{\pi}}))\to Z(\End_{\W(k)[G]}(P_{\bar{\pi}}))$ that sends an endomorphism to its action on $P_{\bar{\pi}}$ is surjective.\end{pro}

\begin{proof} The key point is that $P_{\bar{\pi}}\otimes L$ and $P_0\otimes L$ each contain a single copy of $\St_L$ and that this representation is the only irreducible representation they have in common.  Indeed, $P_0\otimes L$ is a parabolic induction and so, of the representations named in Proposition \ref{CharZeroProjEnv}, only $\St_L$ can admit a nonzero map to or from $P_0\otimes L$.  Furthermore, $P_0\otimes L$ is equal to the parabolic induction of $n\mathbbm{1}$ and so it contains a single copy of $\St_L$.

In particular, $P_{\bar{\pi}}\otimes K$ and $P_0\otimes K$ have only one irreducible $K[G]$-representation, $\St_K$, in common, each representation contains a unique copy of $\St_K$, and $\St_K$ satisfies $$\St_L=\St_K\otimes_KL.$$ Let $f$ be the canonical surjection $P_{\bar{\pi}}\otimes_{\W(k)}K\to\St_K$.  We denote by $M$ the image of $P_{\bar{\pi}}$ in $\St_K$ so that we have the commutative diagram $$\xymatrix{P_{\bar{\pi}}\otimes_{\W(k)}K\ar[rr]^f & &  \St_K \\ P_{\bar{\pi}}\ar@{^{(}->}[u]\ar[rr] & & M\ar@{^{(}->}[u] }$$ and so that $M$ is a $G$-stable $\W(k)$-lattice in $\St_K$.

We claim that $g\in\End_{\W(k)[G]}(P_{\bar{\pi}})$ acts on $M$ by a scalar in $\W(k)$.   Indeed, induces a map on $\St_K$ via restriction.  Since $\St_K$ is absolutely irreducible, this map must be multiplication by a scalar in $K$.  Moreover, $g$ preserves $M$ and so $g$ is a scalar in $K$ which preserves the $\W(k)$-lattice $M$.  We conclude that $g$ is a scalar in $\W(k)$ as claimed.

We thus obtain a map $s:\End_{\W(k)[G]}(P_{\bar{\pi}})\to\W(k)$ which takes a map to the scalar by which it acts on $M$.  Certainly then, for $g\in Z(\End_{\W(k)[G]}(P_{\bar{\pi}}))$, we may view $s(g)$ as an element of $\End_{\W(k)[G]}(P_0)$ and consider the element $$\Phi(g)=\left(\begin{matrix}s(g) & 0 \\ 0 & g \end{matrix}\right)\in\End_{\W(k)[G]}(Q_{\bar{\pi}}).$$

We claim that $\Phi(g)\in Z(\End_{\W(k)}(Q_{\bar{\pi}}))$.  Indeed, for $\psi\in\Hom_{\W(k)[G]}(P_{\bar{\pi}},P_0)$, to say that $\Phi(g)$ commutes with $$\left(\begin{matrix} 0 & \psi \\ 0 & 0 \end{matrix}\right)$$ is to say that for $x\in P_{\bar{\pi}}$, we have $$\psi(g(x))=s(g)\psi(x)$$ in $P_0$. But now since $P_0$ and $P_{\bar{\pi}}$ embed into $P_0\otimes_{\W(k)}K$ and $P_{\bar{\pi}}\otimes_{\W(k)}K$, respectively, it suffices to check this equality over $K$.  Since $P_0\otimes_{\W(k)}K$ and $P_{\bar{\pi}}\otimes_{\W(k)}K$ have only the factor $\St_K$, we have  $\psi=\psi\circ f$.  Since $g$ commutes with $f$ and since $g$ acts by $s(g)$ on $\St_k$, we conclude that $\psi(g(x))=\psi(s(g)x)$ and the equality sought holds trivially.

Likewise, for $\psi\in\Hom_{\W(k)[G]}(P_0,P_{\bar{\pi}})$ the commutativity of $\Phi(g)$ and $$\left(\begin{matrix} 0 & 0 \\ \psi & 0 \end{matrix}\right)$$ is trivial.  Finally, $\Phi(g)$ commutes with a diagonal endomorphism because $\End_{\W(k)[G]}(P_{\bar{\pi}})$ is commutative. Certainly $\Phi(g)$ maps to $g$.\end{proof}

Since we have already given an isomorphism $$e_{\bar{\pi}}Z(\W(k)[G])\to \End_{\W(k)[G]}(Q_{\bar{\pi}}),$$ we conclude that we have maps $$e_{\bar{\pi}}Z(\W(k)[G])\twoheadrightarrow\End_{\W(k)[G]}(P_{\bar{\pi}})\hookrightarrow \End_{L[G]}(P_{\bar{\pi}}\otimes_{\W(k)}L).$$  Denoting this composition by $\delta$, we see that $\End_{\W(k)[G]}(P_{\bar{\pi}})$ is isomorphic to the image of $\delta$.  We have thus reduced our problem to a computation of this image.

\section{Invariants in a Cyclotomic Algebra}

\label{S:Alg} We continue to assume that $n$ is the order of $q$ modulo $\ell$ and we denote by $r$ be the $\ell$-valuation of $q^n-1$.  In this section, we will compute the space of invariants mentioned in Theorem \ref{Main}.  To be precise, for the remainder of the section, we put $$R=\W(k)[X]/(X^{\ell^r}-1)~\textrm{and}~f(X)=X+X^q+\cdots+X^{q^{n-1}}.$$ Then we have a map $$\W(k)[Y]\to R$$ induced by $Y\mapsto f(X)$ and the image of this map is fixed under the transformation $h(X)\mapsto h(X^q)$.  The aim of this section is to show the converse.

We begin by working over certain $\ell$-power cyclotomic extensions of $K$ (which will have an obvious relation to $R$).  Explicitly, for each $0<i\leq r$, fix a primitive $\ell^i$th root unity, $\zeta_i$.  Then we may identify $(\Z/\ell^i\Z)^\times$ with the Galois group of $K(\zeta_i)/K$ in the usual way.  Our first aim is to understand the fixed field, $L_i$, of the subgroup generated of $\Gal(K(\zeta_i)/K)$ generated by $q$. To this end, put $$T=\Z[X]/\left<X^{q^n-1}-1\right>.$$ As above, $a\in\Z$ gives rise to the endomorphism $h(X)\mapsto h(X^a)$ of $T$.  In this way, we get an action of $(\Z/(q^n-1)\Z)^\times$ on $T$.  In particular, the cyclic subgroup generated by $q$ acts on $T$.

\begin{lem}\label{OrbitOrder}  Consider an element of the form $X^a$ in $T$.  If the orbit of $X^a$ under multiplication by $q$ has order strictly less than $n$, then $a$ is divisible by $\ell^r$.\end{lem}

\begin{proof} Suppose the orbit in question has order $b$.  Then $$X^{q^ba}-X^a=X^a(X^{a(q^b-1)}-1)$$ is divisible by $X^{q^n-1}-1$ so that $q^n-1$ divides $a(q^b-1)$.  In particular, since $b<n$, $a$ is divisible by $\Phi_n(q)$ ($\Phi_n(X)$ being the $n$th cyclotomic polynomial) and so by $\ell^r$.\end{proof}

Next, consider the element $$N(X)=(X-1)(X^q-1)\cdots(X^{q^{n-1}}-1)$$ in $T$.  We remark that $N(\zeta_i)$ is the norm of $\zeta_i-1$ along the extension $K(\zeta_i)/L_i$ and so it is a uniformizer for the ring of integers of $L_i$.  Putting $$\omega_i=\zeta_i+\zeta_i^q+\cdots+\zeta_i^{q^{n-1}},$$ we have the following:

\begin{pro}\label{UnityQInv}  Suppose that $0\leq i\leq r$.  Then $L_i=K(\omega_i)$ and the ring of integers of $L_i$ is $\W(k)[\omega_i]$.  Moreover, $\omega_i-n$ is a uniformizer for $\W(k)[\omega_i]$.\end{pro}

\begin{proof} Put $\zeta=\zeta_i$, $\omega=\omega_i$, and $L=L_i$.  In addition, let $S$ be the collection of orbits of powers of $X$ appearing in the expansion of $N(X)$.  We claim that $N(X)$ may be written in the form $$N(X)=\sum_{\O\in S}\left[(-1)^\O\sum_{X^a\in\O}X^a\right],$$ where for each $\O\in S$, $(-1)^\O$ has the value $\pm1$.

To show this claim, we first assume that $q\neq2$.  We see that all the powers of $X$ appearing in the expansion of $N(X)$ are distinct (looking for example at $q$-adic expansions).  Moreover, the degree of $N(X)$ is less than $q^n-1$, and so all of these powers of $X$ are $\Z$-linearly independent in $T$. Since $N(X)$ is preserved under the action of $q$, the sign of any two monomials in an orbit is the same. In the case $q=2$, the argument applies to all of the terms except for the first and the last.  But those terms are each fixed by the action of $q$ and so the claim is trivial.

In particular, $N(1)=\sum_{\O\in S}(-1)^\O|\O|$, but $N(1)=0$.  Thus we have  $$N(\zeta)=\sum_{\O\in S}\left[(-1)^{\O}\left(\sum_{X^a\in\O}\zeta^a-|\O|\right)\right]$$ and so the latter sum is a uniformizer for $L_i$.  Moreover, if $\O\in S$ does not have order $w$, Lemma \ref{OrbitOrder} implies that for any $X^a\in\O$, $\ell^r$ divides $a$.  Hence $\zeta^a$ is 1 and so the corresponding term is zero.  Thus if $S'$ is the collection of orbits of order $n$, we obtain $$N(\zeta)=\sum_{\O\in S'}\left[(-1)^{\O}\left(\sum_{X^a\in\O}\zeta^a-n\right)\right].$$

Since $\zeta$ is an $\ell$th power root of unity, $\zeta$ reduces to 1 in the residue field of $K(\zeta_i)$ and so each term in the sum indexed by $S'$ reduces to 0.  Hence the sum over $S'$ is a sum of elements of positive valuation which add to a uniformizer of $L$. We conclude that at least one of the terms of this sum is a uniformizer for $L$.  That is, we have an $a$ so that $$\zeta^a+\zeta^{aq}+\cdots+\zeta^{aq^{n-1}}-n$$ is a uniformizer for $L$.   This $a$ must be prime to $\ell$ as otherwise the uniformizer will lie in a small extension of $K$.

Thus $a$ is prime to $\ell$.  But the Galois automorphism $\zeta\mapsto\zeta^a$ takes $\omega-n$ to $$\zeta^a+\zeta^{aq}+\cdots+\zeta^{aq^{n-1}}-n.$$ In particular, $\omega-n$ is a uniformizer and so it generates $L$.\end{proof}

\begin{cor}\label{PullBackModL} Over $k$, the polynomial $$g(X)=X+X^q+\cdots+X^{q^{n-1}}-n$$ is divisible by $(X-1)^n$, but not $(X-1)^{n+1}$.\end{cor}

\begin{proof} Let $i$ be the power of $(X-1)$ in the factorization of $g(X)$ modulo $\ell$.  We work over $\W(k)[X]$ and write \begin{equation}\label{ModPoly} g(X)=(X-1)^ig_0(X)+\ell h(X),\end{equation} where $g_0,h\in\W(k)[X]$ and $g_0(1)\neq1$ in $k$.  Put $\zeta=\zeta_r$ and let $\nu$ be the normalized valuation on $K[\zeta]$.  Then $\nu(g_0(\zeta))=0$ and, from Proposition \ref{UnityQInv}, $\nu(g(\zeta))=n$.  On the other hand, $\nu(\ell)=\phi(\ell^r)\geq n$.  Thus \ref{ModPoly} implies that $\nu[(\zeta-1)^i]\geq n$ so that $i\geq n$.

Suppose, on the other hand, that $i>n$.  After evaluation at $\zeta$ the image of the right-hand side of \ref{ModPoly} under $\nu$ is at least $$\min\{i,\ell^{r-1}(\ell-1)\}.$$  Under the assumption that we do not have $r=1$ and $n=\ell-1$, this integer is larger than $n$, a contradiction.

In the case that $n=\ell-1$, we show directly that, over $k$, $g^{(n+1)}(1)\neq0$.  To this end, we have $$g^{(n+1)}(1)=q^{n-1}(q^{n-1}-1)\cdots(q^{n-1}-(n-1))+\cdots+q(q-1)\cdots(q-(n-1)).$$ Since $n$ is even in this case, we may put $n=n'/2$ and get $q^{n'}\equiv-1\mod{\ell}$.  Furthermore, no other power of $q$ between $0$ and $n-1$ is equivalent to $-1$.  In other words, all these powers of $q$ are equivalent to one of $$\{1,2,\ldots,\ell-2\}=\{1,\ldots,n-1\}.$$  Thus the only product in our expansion for $g^{(n+1)}(1)$ which survives is $$q^{n'}\cdots(q^{n'}-(n-1))$$ and it is the product of all the elements of $(\Z/\ell\Z)^\times$, which is nonzero in $k$.\end{proof}

We also denote the minimal polynomial of $\omega_i$ over $\W(k)$ by $m_i(X)$ and we put $m(X)=(X-n)\prod_{i=1}^rm_i(X)$.

\begin{lem}\label{IndividPullBack} Consider the map $\W(k)[Y]\to R$ given by $Y\mapsto f(X)$.  The pre-image of the ideal of $R$ generated by $\phi_{\ell^i}(X)$ is $m_i(Y)$.\end{lem}

\begin{proof} We have the commutative diagram: $$\xymatrix{\W(k)[Y]\ar[r]\ar[dr] & R\ar[d]\\ & \W(k)[\zeta_i]}$$ where $X$ maps to $\zeta_i$ and $Y$ maps to $\omega_i$.  Since $\Phi_{\ell^i}(X)$ generates the kernel of $R\to\W(k)[\zeta_i]$ and $m_i(Y)$ generates that of $\W(k)[Y]\to\W(k)[\zeta_i]$, the claim follows.\end{proof}

We conclude that $m(Y)$ generates the kernel of $\W(k)[Y]\to R$.  Thus, putting $$S=\W(k)[Y]/(m(Y)),$$ we get the injective map $S\to R$ given by $Y\mapsto f(X)$.  We are now in a position to complete the goal of this section:

\begin{thm}\label{qInvCylAlg} The subalgebra of $R$ invariant under $X\mapsto X^q$ is the $\W(k)$-subalgebra generated by $$f(X)=X+X^q+\cdots+X^{q^{n-1}}$$ and so is isomorphic to $$\W(k)[Y]/m(Y).$$\end{thm}

\begin{proof} First of all, inverting $\ell$, the map $$S\otimes_{\W(k)}K\to R\otimes_{\W(k)}K$$ gives an isomorphism of $S\otimes L$ onto the $q$-invariants of $R\otimes K$.  Indeed, the Chinese remainder theorem implies that $R\otimes K$ is the product of the $K[\zeta_i]$ so that the space of $q$-invariants is the product of the $L_i$.  Proposition \ref{UnityQInv} then implies that this space is the product of the $K(\omega_i)$, that is, the image of $S\otimes K$ (again by the Chinese Remainder Theorem).

In particular, if $x\in R$ is a $q$-invariant, $\ell^bx$ lies in the image of $S$ for some sufficiently large $b$.  Thus, it suffices to show that for any $x\in R$ such that $\ell x$ is in the image of $S$, $x$ is in the image of $S$ as well.  This statement, in turn, is equivalent to showing that the map $S/\ell S\to R/\ell R$ is injective.

Now, all the roots of $m(Y)$ are, by definition, Galois conjugates of the various $\omega_i$, all of which reduce to $n$ modulo $\ell$.  Hence, modulo $\ell$, $m(Y)$ is a power of $Y-n$ and we see easily that its degree is $$1+\frac{\ell^r-1}{n}=\left\lceil\frac{\ell^r}{n}\right\rceil.$$  Likewise, modulo $\ell$, $X^{\ell^r}-1$ is equal to $(X-1)^{\ell^r}$.  Thus the statement we need to show is that, under the map, $k[Y]\to k[X]$, induced by $Y\mapsto f(X)$, the pull back of $(X-1)^{\ell^r}$ is equal to the ideal generated by $$(Y-n)^{\left\lceil\frac{\ell^r}{n}\right\rceil}.$$  This statement, however, follows easily from Corollary \ref{PullBackModL}.\end{proof}

\section{The Central Action of the Group Algebra}

\label{S:CentralAction} Once again, $\bar{\pi}$ is the cuspidal representation $\bar{\pi}_1$ of $G$ (so that $n$ is order of $q$ modulo $\ell$).  We put $r=\ord_\ell(q^n-1)$.  As it will be convenient for us to perform explicit calculations with characters, we choose a generator, $\epsilon$, of the $\ell$-torsion part of $\F_{q^n}^\times$.  Recall that we have fixed an identification between $\F_{q^{n!}}$ and its character group, so that $\epsilon$ gives a character, $\theta:\F_{q^n}^\times\to L^\times$, which is trivial on an $\ell$-regular element.

We denote by $\ms{R}$ the representations of $G$ over $L$ given in Proposition \ref{CharZeroProjEnv}.  If $\pi$ is a representation in $S$ then we get a map $\delta_\pi:Z(\W(k)[G])\to L$ which takes an element to its action on $\pi$.  We obtain the map $\delta:Z(\W(k)[G])\to R$, where $$R=\prod_{\pi\in\ms{R}}L.$$  The aim of this section is to compute the image of $\delta$.  This map, $\delta$, is of course the one considered at end of Section \ref{S:ProjEnv} and so this computation will complete our work.

A supercuspidal lift of $\bar{\pi}$ is the supercuspidal representation coming from a $\ell$-power torsion element, $\epsilon^i$, of $\F_{q^n}^\times$, (where $i$ is not divisible by $\ell^r$). We denote this representation by $\pi_i$.  We denote the Steinberg representation $\St_L$ by $\pi_0$. Thus if we let $\mc{I}\subset\Z$ be a collection of representatives for the multiplicative action of $q$ on $\Z/\ell^r\Z$ (and assume that $\mc{I}$ has been chosen to contain 0), we see that $\ms{R}=(\pi_i)_{i\in\mc{I}}$.

Thus we can view $R$ as $\prod_{i\in\mc{I}}L$.  Notationally, we will often write elements of $R$ as ordered pairs $(x,y_i)$ where $x$ is the zeroth (so Steinberg) coordinate and $y_i$ is the $i$th coordinate for $i\neq0$ (the value of $y_i$ may of course depend on $i$).  We will also write $\delta_i$ for $\delta_{\pi_i}$.  Finally, for a conjugacy class $C$ of $G$, we will let $\beta_C=\sum_{t\in C}t$ be the corresponding element of $Z(\W(k)[G])$.

The linearity of trace gives the following easy but important observation:

\begin{lem}\label{ActionTrace} If $C$ is a conjugacy class of $G$, then
$$\delta_i(\beta_C)=\frac{|C|}{\dim\pi_i}\Tr|_{\pi_i}(C).$$\end{lem}

Lemma \ref{ActionTrace} indicates that it will important for us to understand the characters of the $\pi_i$, which we will now record.  First of all, we call a conjugacy class of $G$ \emph{primary} if the corresponding minimal polynomial is a power of a single irreducible polynomial.  For $i\neq 0$, the character of the representation $\pi_i$ is equal to zero on all classes that are not primary. For a primary class, $C$, we choose an eigenvalue $t$ of $C$, denote the degree of $t$ over $\F_q$ by $a$ and let $x$ be the number of Jordan blocks in a matrix in $C$ in Jordan canonical form.  On such $C$, the character of $\pi_i$ is $$(-1)^{n-x}(q^a-1)(q^{2a}-1)\cdots(q^{(x-1)a}-1)[\theta^i(t)+\theta^{iq}(t)+\cdots+\theta^{iq^{a-1}}],$$ (see p.275 of \cite{DJ}; note there is a typographical error in the formula given there).  In particular, the dimension of $\pi_i$ does not depend on $i$ and it is equal to $$(q^{n-1}-1)(q^{n-2}-1)\cdots(q-1).$$

The Steinberg character can be defined for any finite reductive group over a finite field of characteristic $p$.  Its character vanishes at group elements with of order divisible $p$.  Thus in our case, the Steinberg character vanishes on a matrix that is not diagonalizable.  For an element, $g$, of order prime to $p$, the Steinberg character at $g$ is, up to sign, the order of a $p$-Sylow subgroup of the centralizer of $g$ (see Theorem 3.1 of \cite{Steinberg} and also \cite{SteinHis}).  We have of course recorded the Steinberg character only up to sign, but that will be sufficient for our purposes.  Note that the dimension of the Steinberg character is $q^{n(n-1)/2}$. Thus the dimension of each representation in $\ms{R}$ is a unit in $\W(k)$.

Lemma \ref{ActionTrace} also indicates that we will need some control on the order of conjugacy classes in $G$.  The relevant observation is the following:

\begin{lem}\label{CentOrder} Let $C$ be the conjugacy class in $G$.  Then unless $C$ is primary and diagonalizable, we have $\ord_\ell(|C|)=r$.\end{lem}

\begin{proof} Barring the cases mentioned, it suffices to show that the order of the centalizer of an element in $C$ is not divisible by $\ell$ (since $\ord_\ell(|G|)=r$).  This fact, however, follows from a discussion on pages 409 - 410 of \cite{Green} regarding the order of these centralizers (note there is a typographical error in the centered equation at the bottom of p. 409: Green's explanation footnote makes the correct equation clear).\end{proof}

We denote by $S$ the $\W(k)$-subalgebra of $R$ generated by the element $(\ell^r,0)$.  In other words, $S$ is the subalgebra of $R$ consisting of elements of the form $(x,y)$ (so that the $i$th coordinate does not depend on $i$ for $i\neq0$) with $x\equiv y\mod{\ell^r}$.  The subalgebra $S$ will play a key role in our calculations as we will show that $\delta(\beta_C)\in S$ for many of the conjugacy classes, $C$, of $G$ (to be precise, this containment will hold for all those conjugacy classes of $G$ on which all of our supercuspidal characters coincide).  We being by showing that $S$ is actually contained in the image of $\delta$.

\begin{pro}\label{Realized} The element $(\ell^r,0)\in R$ is contained in the image of $\delta$.\end{pro}

\begin{proof} Consider the class, $C$, corresponding to a single $n\times n$ Jordan block of eigenvalue 1.  The Steinberg character vanishes on this class. Since, for $i\neq0$, the character of $\pi_i$ is $\pm1$ (the sign being independent of $i$), the claim follows from Lemmas \ref{ActionTrace} and \ref{CentOrder}.\end{proof}

We will now show that many of the $\delta(\beta_C)$ lie in $S$.

\begin{pro}\label{CuspVanish} Let $C$ be a conjugacy class which is not primary. Then $\delta(\beta_C)\in S$.\end{pro}

\begin{proof} Since a supercuspidal character vanishes on $C$, it suffices to show that $\delta_0(\beta_C)\in\ell^r\W(k)$.  To this end, we note that Lemma \ref{CentOrder} implies that $|C|$ is divisible by $\ell^r$.  Since the character of $\pi_0$ lies in $\Z\subset\W(k)$, the claim follows from Lemma \ref{ActionTrace}.\end{proof}

We next consider primary classes admitting an eigenvalue that does not have degree $n$.  If the class is not diagonalizable, the argument is straight forward.

\begin{pro}\label{NotDivisible} Suppose that $C$ is primary with an eigenvalue of degree less than $n$ over $\F_q$.   Suppose also that $C$ is not diagonalizable.  Then $\delta(\beta_C)\in S$.\end{pro}

\begin{proof} Since $C$ is not diagonalizable, the Steinberg character vanishes on $C$.  Furthermore, if $t$ is an eigenvalue of $C$, $\ell$ does not divide the order of $t$ so that $\theta^{i}(t)=1$ for any $i$.  Hence, for $i\neq0$, the character of $\pi_i$ on $C$ is an element of $\W(k)$ which does not depend on $i$.  Lemma \ref{CentOrder} thus shows that $\delta_i(\beta_C)\in\ell^r\W(k)$ (using Lemma \ref{ActionTrace} as usual) and the claim follows.\end{proof}

To handle the diagonalizable case, we will need to handle the characters a bit more delicately.  In particular, we are in need of the following routine result:

\begin{lem}\label{Signs}  Suppose that $v,d\in\N$ with $vd=n$.  Then $$\frac{(q^d-1)(q^{2d}-1)\cdots(q^{(v-1)d}-1)}{v}\equiv q^{n(v-1)/2}\mod{\ell^r}.$$
\end{lem}

\begin{proof} First suppose that $\zeta$ is a primitive $v$th root of unity over $\Q$ and consider the polynomial $$\sum_{i=0}^{v-1}X^i=\frac{X^v-1}{X-1} =\prod_{i=1}^{v-1}(X-\zeta^i)=(-1)^{v-1}\prod_{i=1}^{v-1}(\zeta^i-X).$$ Evaluating at $1$ shows that $$(\zeta-1)(\zeta^2-1)\cdots(\zeta^{v-1}-1)=(-1)^{v-1}v.$$  In other words, the polynomial \begin{equation}\label{PolyExp}(X-1)(X^2-1)\cdots(X^{v-1}-1)-(-1)^{v-1}v\end{equation} is divisible by $\Phi_v(X)$.

But now, since $\Phi_n(X)$ divides $\Phi_v(X^d)$, $\ell^r$ divides $\Phi_v(q^d)$.  Evaluating the polynomial in \eqref{PolyExp} at $q$, we conclude that $$(q^d-1)(q^{2d}-1)\cdots(q^{(v-1)d}-1)\equiv(-1)^{v-1}v\mod{\ell^r}.$$ A routine argument shows that $q^{n(v-1)/2}\equiv(-1)^{v-1}\mod{\ell^r}$ and the claim follows.\end{proof}

Using Lemma \ref{Signs}, we may prove the following:

\begin{pro}\label{NotDivisible2} Suppose that $C$ is a diagonalizable primary class whose eigenvalues have degree less than $n$ over $\F_q$.  Then $\delta(\beta_C)\in S$.\end{pro}

\begin{proof} We again let $t$ be an eigenvalue of $C$.  Suppose that $a$ is the degree of $t$ and that $b=n/a$ (so that $b$ is the number of Jordan blocks in $C$).  Again since $\theta^i(t)=1$ for any $i$, Lemma \ref{ActionTrace} implies that $$\delta(\beta_C)=|C|\left(\pm\frac{q^{b(b-1)/2}}{q^{n(n-1)/2}},
\frac{(-1)^{n-b}a(q^a-1)(q^{2a}-1)\cdots(q^{(b-1)a}-1)}{(q-1)(q^2-1)\cdots(q^{n-1}-1)}\right).$$
In particular, for $i\neq0$, $\delta_i(\beta_C)$ does not depend on $i$. Moreover, two applications of Lemma \ref{Signs} (one for $v=n$ and one for $v=b$) show that $\delta_i(\beta_C)\equiv\pm\delta_0(\beta_C)\mod{\ell^r}$.  But because modular reductions of $\pi_i$ and $\pi_0$ must both contain $\bar{\pi}$ as a constituent, $\beta_C$ must act on $\pi_i$ and $\pi_0$ by the same scalar modulo $\ell$.  Since $\ell\neq 2$, we conclude the coordinates of $\delta(\beta_C)$ are equivalent modulo $\ell^r$.  \end{proof}

To summarize our work so far, Propositions \ref{Realized}, \ref{CuspVanish}, \ref{NotDivisible}, and \ref{NotDivisible2} show that the image of $\delta$ is the $\W(k)$-subalgebra of $R$ generated by $(\ell^r,0)$ and the elements $\delta(\beta_C)$, where $C$ is a primary conjugacy class admitting an eigenvalue, $t$, of degree $n$ over $\F_q$.

In the latter case, Lemma \ref{ActionTrace} shows that $$\delta(\beta_C)=|C|\left(\frac{\pm1}{q^{n(n-1)/2}},\frac{(-1)^{n-1}[\theta^i(t)+\theta^{iq}(t)+\cdots+\theta^{iq^{n-1}}(t)]}{(q-1)\cdots(q^{n-1})}\right),$$
as the order of the centralizer of an element in $C$ is $q^n-1$ (and so a $p$-Sylow group of this centralizer is trivial).  Since $|C|$ is a unit in $\W(k)$, we see that, up to multiplication by a unit in $\W(k)$, $\delta(\beta_C)$ is
$$\left(\frac{\pm(q-1)\cdots(q^{n-1})}{q^{n(n-1)/2}},\theta^i(t)+\theta^{iq}(t)+\cdots+\theta^{iq^{n-1}}(t)\right).$$
Lemma \ref{Signs} then shows that we may use an appropriate element of $S$ to obtain the element $$\left(\pm n,\theta^{i}(t)+\theta^{iq}(t)+\cdots+\theta^{iq^{n-1}}(t)\right).$$  Since these coordinates must be equivalent modulo $\ell$, the element must in fact be $$\gamma_t=\left(n,\theta^i(t)+\theta^{iq}(t)+\cdots+\theta^{iq^{n-1}}(t)\right).$$ Thus we see that the image of $\delta$ is generated by these $\gamma_t$ and $(\ell^r,0)$.

To simplify our notation, put $\zeta=\theta(\epsilon)$ so that $\zeta$ is a primitive $\ell^r$th root of unity and put $$\gamma=\gamma_{\epsilon}=(n,\zeta^i+\zeta^{iq}+\cdots+\zeta^{iq^{n-1}}).$$
We will show that the image of $\delta$ is generated by $\gamma$.

\begin{pro} If $t\in\F_{q^n}^\times$ has degree $n$ over $q$, then $\gamma_t$ is contained the $\W(k)$-subalgebra of $R$ generated by $\gamma$.\end{pro}

\begin{proof} By definition of $\epsilon$, we may find a $j$ so that the $\ell$-power torsion part of $t$ is $\epsilon^j$.  This gives $\theta(t)=\theta(\epsilon^j)=\zeta^j$.  Hence
$$\gamma_t=(n,\zeta^{ij}+\zeta^{ijq}+\cdots+\zeta^{ijq^{n-1}}).$$
Correspondingly, consider the polynomial $$\alpha(X)=X^j+X^{jq}+\cdots+X^{jq^{n-1}}$$ in $W(k)[X].$  Modulo $(X^{\ell^r}-1)$, this polynomial is fixed under $X\mapsto X^q$ and so Theorem \ref{qInvCylAlg} shows that we may find a polynomial $h(X)\in\W(k)[X]$ such that $$\alpha(X)\equiv h(X+X^{q}+\cdots+X^{q^{n-1}})\mod{(X^{\ell^r}-1)}.$$  Evaluating at $\zeta^i$ for each $i$ (including $i=0$), we conclude that $\gamma_t=h(\gamma).$\end{proof}

Thus the image of $\delta$ is generated by $\gamma$ and $(\ell^r,0)$.  In order to show that $(\ell^r,0)$ is not a necessary generator, put $\mc{I}'=\mc{I}-\{0\}$ and $$g(X)=\prod_{i\in \mc{I'}}[X-(\zeta^i+\zeta^{iq}+\cdots+\zeta^{iq^{n-1}})].$$

\begin{pro} The element $(\ell^r,0)$ is contained in the $\W(k)$-subalgebra generated by $\gamma$.\end{pro}

\begin{proof} By construction, $g(\gamma)=(a,0)$ where $$a=\prod_{i\in\mc{I}'}[n-(\zeta^i+\zeta^{iq}+\cdots+\zeta^{iq^{n-1}})].$$ Note that every nontrivial $\ell^r$th root of unity is represented in exactly one factor of $a$.  On the other hand, by Proposition \ref{UnityQInv}, if $i\in \mc{I}'$ is such that $\zeta^i$ is a primitive $\ell^s$th root of unity, the $\W(k)$-valuation of the factor coming from $i$ is $n/\phi(\ell^s)$.  Thus $\ord_\ell(a)=r$ and, up to multiplication by a unit in $\W(k)$, $g(\gamma)$ is $(\ell^r,0)$.\end{proof}

We remark that the minimal polynomial of $\gamma$ is $f(X)=(X-n)g(X)$.  We may thus show the following key result:

\begin{thm} The image of $\delta:Z(\W(k)[G])\to R$ is isomorphic as a $\W(k)$-algebra to $$\W(k)[Y]/(f(Y)).$$  That is, it is isomorphic to the invariants of $$\W(k)[X]/(X^{\ell^r}-1)$$ under the action of $q$.\end{thm}

\begin{proof} Since $\gamma$ generates the image of $\delta$ and has minimal polynomial $f$, the claim follows from
Theorem \ref{qInvCylAlg}.\end{proof}

We have thus complete our calculation.  To summarize, we have shown the following:

\begin{thm}\label{Main} Suppose that $k$ is a finite field of characteristic $\ell$. Let $q$ be a power of a prime distinct from $\ell$.  Denote the order of $q$ modulo $\ell$ by $w$ and put $r=\ord_\ell(q^w-1)$. Suppose that $\bar{\pi}$ is an irreducible cuspidal representation of $G=\GL_n(\F_q)$ over $k$ of degree $d<n$.  Let $P_{\bar{\pi}}$ be the projective envelope of $\bar{\pi}$ in the category of $\W(k)[G]$-modules.  Then, under the assumptions that $2\leq n<\ell$ and that $k$ is large enough to contain the $\ell$-regular $|G|$th roots of unity, $\End_{\W(k)[G]}(P_{\bar{\pi}})$ is isomorphic to $$\W(k)[Y]/\prod_{\zeta\in \mc{S}}[Y-(\zeta^i+\zeta^{iq^d}+\cdots+\zeta^{iq^{n-d}})],$$ where $\mc{S}$ is a collection of representatives for the action of $q^d$ on the $\ell^r$th roots of unity in $\W(k)$.  That is, $\End_{\W(k)[G]}(P_{\bar{\pi}})$ is isomorphic to the ring of invariants of $$\W(k)[X]/(X^{\ell^r}-1)$$ under the action $X\mapsto X^{q^d}$.\end{thm}

\begin{rem}An abstract isomorphism is useful, but it is typically more advantageous to have a specific isomorphism.  In this present case, specifying the isomorphism above amounts to giving the action of the generator $Y$ on the generic characteristic zero representations in the block coming from $\bar{\pi}$.  Tracing through our calculations (or from \cite{thesis}), one can see that, if $\bar{\pi}=\bar{\pi}_s$, our isomorphism may be chosen in such a way that the $Y$ acts on a supercuspidal representation $\pi_\tau$ (that lifts $\bar{\pi}_s$) by $$\theta(\tau)+\theta(\tau^{q^d})+\cdots+\theta(\tau^{q^{n-d}})$$ and that it acts on the generalized Steinberg representation associated to $s$ by $n/d$.\end{rem}

\section{Relation with a Conjecture of Helm}

\label{S:Helm}In this section, we discuss the relationship between this work and the conjecture of Helm discussed in Section \ref{S:Introduction}.  In the process, we will sketch some of Helm's work on this conjecture (this work will appear in an upcoming paper, \cite{Helm2}).  Recall from Section \ref{S:Introduction} that $F$ is a $p$-adic field and that $\bar{\rho}$ is a semisimple representation over $k$ of the Weil Group, $\W_F$, that corresponds to a representation, $\bar{\pi}$, of $\GL_n(F)$ which is cuspidal but not supercuspidal.  Denote the inertia subgroup of $W_F$ by $I$, let $I^{(\ell)}$ be the prime to $\ell$ part of $I$, and let $I_\ell$ be the corresponding quotient.  We continue to assume that $\ell>n$.

The simplest possibility is the case in which $n$ is the order of $q$ modulo $\ell$ and $$\bar{\rho}=\bar{\mathbbm{1}}\oplus\bar{\omega}\oplus\bar{\omega}^2+\cdots+\oplus\bar{\omega}^{n-1},$$
where $\bar{\omega}$ is the cyclotomic character of $F$ over $k$.  In this case, a deformation of $\rho$ is trivial on $I^{(\ell)}$ and so must factor through $$\W_F\to I_\ell\rtimes\Z.$$

If $\Psi$ is a topological generator for $I_\ell$ and $\Fr$ is a Frobenius element, we have the relation $\Fr\Psi\Fr^{-1}=\Psi^q$.  Hence to choose a framed deformation of $\bar{\rho}$ is to choose matrices that satisfy this relation (and reduce appropriately to $\bar{\rho}$).   Thus, putting $$R_{q,n}=\W(k)[\Fr,\Psi]/\left<\Fr\Psi \Fr^{-1}=\Psi^q\right>,$$ where the variables $\Fr$ and $\Psi$ are $n\times n$ matrices of indeterminates,  we can see that $R_{\bar{\rho}}^\boxempty$ is the completion of $R_{q,n}$ at $\m_{\bar{\rho}}$, the maximal ideal coming from $\bar{\rho}$.

We now consider an important subalgebra of $R_{q,n}$.  Recall that $$f(X)=\prod_{i\in \mc{I}}[X-(\zeta^i+\zeta^{iq}+\cdots+\zeta^{iq^{n-1}})].$$

\begin{lem}\label{FinalLemma} Let $T_1,\ldots,T_n$ be the coefficients of the characteristic polynomial of the matrix $\Fr$ and let $Y=\Tr(\Psi)$.  Then the subalgebra of $R_{q,n}$ generated by $Y$, $T_n^{-1}$, and the $T_i$ is isomorphic to $$\W(k)[T_1,\ldots,T_n^{\pm1},Y]/(\left<f(Y)\right>+\left<Y-n\right>\left<T_1,\ldots,T_n\right>).$$\end{lem}

\begin{proof} From work in to be published in \cite{Helm2}, $R_{q,n}$ is reduced and $\ell$-torsion free.  Thus to show the relations we have claimed, it suffices to show that they hold on the $\Omega$-points of $[R_{q,n}]_{\m_{\bar{\rho}}}$ where $\Omega$ is an algebraically closed extension of $\Q_\ell$.  Let $x$ be such a point.  Then $\Psi_x=x(\psi)$ and $\Fr_x=x(\Fr)$ are $n\times n$ matrices over $\Omega$ such that $\Fr_x\Psi_x\Fr_x^{-1}=\Psi_x^q$ and $\Psi_x$ reduces to the identity modulo $\ell$.

In particular, the eigenvalues of $\Psi_x$ are closed under exponentiation by $q$.  Since $\Psi_x$ reduces to the identity, we conclude that all the eigenvalues are $\ell^r$th roots of unity. If the eigenvalues are not trivial, we may assume without loss of generality that $\Psi_x=\diag(\zeta,\zeta^q,\ldots,\zeta^{q^{n-1}}),$ where $\zeta$ is a nontrivial $\ell^r$th root of unity.  Thus $f(\Tr(\Psi_x))=0$ and the relation between $\Fr_x$ and $\Psi_x$ shows that $\Fr_x$ has the form $$\left(\begin{matrix} 0 & * & 0 & \cdots & 0\\ 0 & 0 & * & \cdots & 0 & \\ 0 & 0 & 0 & \cdots & 0 \\ \vdots & \vdots & \vdots & \ddots & \vdots \\ * & 0 & 0 & \cdots & 0\end{matrix}\right).$$ Thus, aside from the determinant, the coefficients of the characteristic polynomial of $\Fr_x$ vanish and the relations $\left<T_1,\ldots,T_{n-1}\right>$ are satisfied.  If, on the other hand, the eigenvalues of $\Psi_x$ are all equal to 1, we have $\Tr(\Psi_x)-n=0$ (so that in particular $f(\Tr\Psi_x)=0$).  Thus the relations are satisfied in either case.

Conversely, $f$ is the minimal polynomial of $Y$ because we can always find a point $x$ such that the minimal polynomial of $x(\Psi)$ is a given factor $f$.  Indeed, we need only choose an appropriate $\zeta$, put $x(\Psi)=\diag(\zeta,\zeta^q,\ldots,\zeta^{q^{n-1}}),$ and take $\Fr_x$ to be the appropriate permutation matrix.  If, then, $h$ is a polynomial in $$[\W(k)[Y]/f(Y)][T_1,\ldots,T_n^{\pm1}]$$ that vanishes as an element of $R_{q,n}$, we know that, for any appropriate $\Omega$, $h$ must vanish on an $\Omega$ point of the form $(\Fr_x,\mathrm{I})$.  Thus $h$ is divisible by $Y-n$.  Evaluating $Y$ at a nontrivial $\ell^r$th root of unity and using that fact that $T_n$ is a unit in $R_q,n$, we see that any monomial appearing in $h$ must also be divisible by $T_i$ for some $i<n$.\end{proof}

The discussion of Section \ref{S:Introduction} characterizes $A_\pi$ in terms of $\End(P_{\bar{\pi}})$ and the ideal $I_0$ described there.  The remark following Theorem \ref{Main} shows that $I_0$ is generated by $Y-n$, where $Y$ is our generator for $\End(P_{\bar{\pi}})$.  Hence applying our result, Theorem \ref{Main}, and Lemma \ref{FinalLemma} we see that $A_\pi$ is isomorphic to the subalgebra of $(R_{\bar{\rho}}^\boxempty)^{\GL_n}$ generated by the trace of $\Psi$ and the coefficients of the characteristic polynomial of $\Fr$.  To be explicit, this isomorphism takes our generator for $\End(P_{\bar{\sigma}})$ to the trace of $\Psi$ and the $T_i$ in Helm's expression for $A_{\bar{\pi}}$ to the coefficients of the characteristic polynomial of $\Fr$.  In particular, our calculations are compatible with the conjecture of Helm. Certainly, one would need to know that this map extends to an isomorphism on all of $(R_{\bar{\rho}}^\boxempty)^{\GL_n}$ and that the isomorphism may be given in a way that is compatible with local Langlands.

In the general case, $\bar{\rho}$ has the form $$\bar{\rho}=\bar{\rho}_0\oplus\bar{\omega}\bar{\rho}_0\oplus\cdots\oplus \bar{\omega}^{e-1}\bar{\rho}_0,$$ where $\bar{\rho}_0$ is an irreducible representation and $e$ is the order of the orbit of $\bar{\rho}_0$ under the action of $\bar{\omega}$.  In the general case, $e$ is not necessarily equal to $w$, the order of $q$ modulo $\ell$, but we do of course have $e|w$.   If $\bar{\pi}_0$ is the supercupsidal representation corresponding to $\bar{\rho}_0$, the supercuspidal support of $\bar{\pi}$ is $\sum_{i=0}^{e-1}\bar{\omega}^i\bar{\pi}_0$.

This more general situation is a bit more complex for a couple of reasons.  Firstly, Helm's isomorphism $$A_{\bar{\pi}}\to\End(P_{\bar{\sigma}})[T_1,\ldots,T_{e-1},T_{e}^{\pm1}]/\left<T_1,\ldots,T_{e}\right>\cdot I_0$$ depends on the choice of a certain type of lift of $\bar{\pi}_0$ to a characteristic zero representation.  Furthermore,  a deformation of $\bar{\rho}$ need not be trivial on $I^{(\ell)}$.  To handle this latter difficulty, Helm implements a construction which we now sketch (Helm's construction also accounts for the choice of lift).

Since the order of $I^{(\ell)}$ is prime to $\ell$, the $k$-module $\bar{\rho}_0|_{I^{(\ell)}}$ is semisimple and we may write $$\bar{\rho}_0|_{I^{(\ell)}}=\bar{\tau}_1\oplus\cdots\oplus\bar{\tau}_r,$$ where $\bar{\tau}_i$ is irreducible.  Certainly $\W_F$ acts on the $\bar{\tau}_i$ by conjugation and we denote by $\W_E$ the stabilizer of $\bar{\tau}_1$.  Again since the order of $I^{(\ell)}$ is prime to $\ell$, there is a unique lift, $\tau_1$, of $\bar{\tau}_1$ to $\W(k)$.  Moreover:

\begin{lem}\label{CHT1} There is a unique extension of $\tau_1$ to $I\cap\W_E$ with prime to $\ell$ determinant.  $\tau_1$ extends to a representation of $\W_E$.\end{lem}

\begin{proof} See Lemma 2.4.11 of \cite{CHT}.\end{proof}

We note that the extension of $\tau_1$ to $\W_E$ is not unique, but we will choose such an extension and denote it by $\tau_1$ as well.  We remark this choice of extension corresponds to the lift of $\bar{\pi}_0$ that must be chosen to give Helm's isomorphism on $A_{\bar{\pi}}$.

Let $\rho^\boxempty$ be the universal framed deformation of $\bar{\rho}$.   Then, since we have extended $\tau_1$ to $\W_E$, $\W_E$ acts on the free $R_{\bar{\rho}}^\boxempty$-module $$\rho^\dagger=\Hom_{I^{(\ell)}}(\tau_1,\rho^\boxempty)$$ and $I^{(\ell)}$ acts trivially.  Thus we may pick $\Psi_E$ and $\Fr_E$ for $\W_E$ as we did for $\W_F$ and proceed as in the first case (that is, we let $\Psi_E$ be a topological generator of prime to $\ell$ part of inertia for $\W_E$ and we let $\Fr_E$ be a Frobenius element of $\W_E$). In this case, we have the relation $\Fr_E\Psi_E\Fr_E^{-1}=\sigma^{q^{n/e}}$.  Thus we get a map $$R_{q^e,n/e}\to R_{\bar{\rho}}^\boxempty$$ that takes the entries of $\Fr$ and $\Psi_E$ to those of $\rho^\dagger(\Fr_E)$ and $\rho^\dagger(\Psi_E)$, respectively.

\begin{lem} The map $R_{q^e,n/e}\to R_{\bar{\rho}}^\boxempty$ is injective and takes $\Gl_e$-invariants to $\Gl_n$-invariants.\end{lem}

\begin{proof}  Again these are results to be published (\cite{Helm2}).  Injectivity follows essentially from the fact that one may recover a deformation, $\rho$, of $\bar{\rho}$ from $\Hom_{I^{(\ell)}}(\tau_1,\rho)$.  This fact is
Lemma 2.4.12 and Corollary 2.4.13 of \cite{CHT}.\end{proof}

Thus, applying our work in the first case and Theorem \ref{Main} we see that:

\begin{thm} The coefficients of the  characteristic polynomials of $\Fr_E$ and $\Psi_E$ and generate a subalgebra of $(R_{\bar{\rho}}^\boxempty)^{\GL_n}$ isomorphic to $A_{\bar{\pi}}$.\end{thm}

Again one would want to see that this isomorphism onto a subalgebra is a full isomorphism satisfying Helm's conjecture.  One would also want to drop the assumption that $\ell>n$.  More generally, we have restricted to the case where our semi-simple representation corresponds to a cuspidal representation and one would want to consider the more general case.

\section*{Acknowledgements} 

The author is grateful to David Helm for his endless guidance and patience.  He also wishes to thank Keenan Kidwell for many productive and enlightening discussions.  Finally, he would like to think Mathoverflow user Jef for his post on Mar. 18th which pointed to a result that greatly simplified the proof.

\bibliography{Sources}

\end{document}